\documentclass[12pt]{amsart}
\usepackage{preamble}
\usepackage{tabularx}
\title{Length-Four Pattern Avoidance in Inversion Sequences}
\author{Letong Hong and Rupert Li}
\date{\today}

\begin{document}
\begin{abstract}
Inversion sequences of length $n$ are integer sequences $e_1,\ldots ,e_n$ with $0\le e_i<i$ for all $i$, which are in bijection with the permutations of length $n$.
In this paper, we classify all Wilf equivalence classes of pattern-avoiding inversion sequences of length-4 patterns except for one case (whether 3012 $\equiv$ 3201) and enumerate some of the length-4 pattern-avoiding inversion sequences that are in the OEIS.
\end{abstract}
\maketitle

\section{Introduction}\label{section: introduction}
Pattern avoidance for permutations is a robust and well-established branch of enumerative combinatorics.
We refer readers to Stanley \cite{stanley2010survey} for an overview of this field, and to Simion and Schmidt \cite{simion1985restricted} in 1985 for the first systematic study of pattern avoidance on permutations.
Classical pattern avoidance represents permutations using one-line notation $\pi=\pi_1\cdots\pi_n$; an alternative representation for permutations is using \emph{inversion sequences} $e=e_1\cdots e_n$, sequences of integers such that $0 \leq e_i < i$ for all $i$.
Inversion sequences are in natural bijection with permutations via the well-known Lehmer code \cite{lehmer1960teaching}, an example of an inversion table: one can biject an inversion sequence $e$ to a permutation $\pi$ via ensuring that for each $i$, there exist $e_i$ values $j<i$ such that $\pi_j > \pi_i$.
Inversion sequences have been studied in many contexts and fields, not just pattern avoidance; for example, see Savage and Schuster \cite{savage2012ehrhart}. 

The study of pattern avoidance on inversion sequences was concurrently initiated by Mansour and Shattuck \cite{MansourShattuck} in 2015 and Corteel, Martinez, Savage, and Weselcouch \cite{corteel2016patterns} in 2016.
The former obtained the explicit number and/or generating function of inversion sequences avoiding any element of $S_3$; the latter further enumerated the number of pattern-avoiding sequences for all patterns of length 3 and related these quantities to well-known combinatorial sequences including the Bell numbers, Euler up/down numbers, Fibonacci numbers, and Schr\"oder numbers.
For patterns of length 4, Chern \cite{Chern} proved the exact formula for 0012-avoiding inversions, answering a conjecture by Lin and Ma (see the end of \cite{lin2020inversion}) in 2020.
However, the enumeration, or even the determination of the Wilf equivalence classes, for all other patterns of length 4 remains open.
For simultaneous avoidance of multiple patterns, Lin and Yan \cite{LinYan} in 2020 studied inversion sequences avoiding certain combinations of two length-3 patterns by establishing correspondences with objects enumerated by the Bell numbers, Fishburn numbers, powered Catalan numbers, semi-Baxter numbers, and 3-noncrossing partitions. In 2018,
Martinez and Savage \cite{Martinez2018PatternsII} rephrased and generalized the question, investigating the avoidance of triples of binary relations, that is, no simultaneous appearances $e_iR_1e_j$, $e_jR_2e_k$, and $e_iR_3e_k$ are allowed to appear with $i<j<k$ for some given $R_1,R_2,R_3\in\{<,>,\le,\ge,=,\neq,-\}$.
On the other hand, in 2019 Auli and Elizalde \cite{AuliElizalde} enumerated the length-3 consecutive pattern-avoiding inversion sequences as well as classified consecutive patterns up to length 4 according to the corresponding Wilf equivalence relations.
In a following 2021 paper \cite{AuliElizaldeWilf}, the same authors gave a complete list of generalized Wilf equivalences between hybrid vincular patterns of length 3, completing the classification of Wilf equivalence classes for all vincular patterns of length 3.
They further built on Martinez and Savage's framework and extended the enumeration to inversion sequences avoiding $e_iR_1e_{i+1}R_2e_{i+2}$ configurations \cite{AuliElizaldeII} in 2019.

Our main result classifies all Wilf equivalence classes for length-4 patterns, except one unresolved case of 3012 possibly belonging to the last class, as demarcated below by a question mark.
A computer search for lengths $n \leq 10$ demonstrates that no other Wilf equivalences are possible: in particular, 2001 agrees with the second Wilf equivalence class $2110\equiv2101\equiv2011$ for all lengths $n\leq 9$, but diverges at $n=10$.
\begin{equation}\label{eq:wilf}
    \begin{gathered}
        1011 \equiv 1101 \equiv 1110 \\
        2110 \equiv 2101 \equiv 2011 \\
        0221 \equiv 0212 \\
        0312 \equiv 0321 \\
        1102 \equiv 1012 \\
        2201 \equiv 2210 \\
        2301 \equiv 2310 \\
        3201 \equiv 3210 \overset{?}{\equiv} 3012.
    \end{gathered}
\end{equation}
\begin{theorem}\label{theorem:main_classification}
Length-4 patterns satisfy the Wilf equivalences listed in \cref{eq:wilf}, with possible exception as demarcated with a question mark.
\end{theorem}

The paper is organized as follows.
In \cref{section: preliminaries}, we introduce necessary definitions and notation.
In \cref{section: Wilf 4}, we establish the aforementioned equivalences using techniques including double induction and direct characterization, together with explicitly constructed correspondences.
In \cref{section: enumeration 4}, we enumerate the 0000 and 0111-avoiding inversion sequences.
Aside from 0012 as addressed by Chern \cite{Chern}, the only other length-4 pattern that is on the OEIS \cite{sloane2018line} is 0021, whose enumeration we leave as an open question, which after the writing of the original version of this paper has been resolved (see \cref{remark:0021 solved}).

\section{Preliminaries}\label{section: preliminaries}

For a positive integer $n$, let $[n]$ denote the set $\{1,2,\dots,n\}$.
An inversion sequence of length $n$ is a sequence $e=e_1\cdots e_n$ of integers such that $0 \leq e_i < i$ for all $i\in[n]$.
We denote the length of $e$ by $|e|=n$.
The set of inversion sequences of length $n$ is denoted by $\I_n$, where we use the convention that $\I_0$ contains exactly one sequence, the empty sequence.

Two sequences of integers $\pi=\pi_1\cdots\pi_k$ and $\sigma=\sigma_1\cdots\sigma_k$ of the same length are said to be \emph{order isomorphic}, denoted $\pi\sim\sigma$, if both $\pi_i R \pi_j$ and $\sigma_i R \sigma_j$ have the same relation $R\in\{<,=,>\}$, for all $1 \leq i,j \leq n$.
For example, $0212\sim 5969$.
A \emph{pattern} refers to such a sequence of integers $\pi=\pi_1\cdots\pi_k$.

For an inversion sequence $e \in \I_n$ and a pattern $\pi=\pi_1\cdots\pi_k$ for $k \leq n$, we say $e$ \emph{contains} $\pi$ as a pattern if there is a not necessarily consecutive subsequence $e'$ of $e$ with $|e'|=k$ such that $e' \sim \pi$.
For $S \subseteq [n]$, we let $e_S$ denote the subsequence of $e$ consisting of the elements $e_i$ for $i \in S$, sorted in ascending order of $i$.
Using this notation, $e$ contains $\pi$ if there exists $S \subseteq[n]$ with $|S|=k$ such that $e_S \sim \pi$.
If $e$ does not contain $\pi$, it is said to \emph{avoid} $\pi$.
In particular, if $|e|<|\pi|$, we also say $e$ avoids $\pi$.
The same definition can be used to define pattern avoidance on permutations.

The \emph{avoidance class} of $\pi$ is
\[ \I_n(\pi) = \{e \in \I_n\mid e \text{ avoids } \pi\}. \]
We say two patterns $\pi$ and $\sigma$ are \emph{Wilf equivalent}, denoted $\pi\equiv\sigma$, if for all $n \geq 1$, we have $|\I_n(\pi)|=|\I_n(\sigma)|$.

We now define the following generalization of an inversion sequence, as originally introduced by Savage and Schuster \cite{savage2012ehrhart}.
\begin{definition}
For a finite set of positive integers $S\subset\Z_+$ enumerated in increasing order $s_1<\cdots<s_n$, an \emph{$S$-inversion sequence} is a sequence $e=e_1\cdots e_n$ of length $n$ such that $0 \leq e_i < s_i$ for all $i\in [n]$.
The set of $S$-inversion sequences is denoted by $\I_S$.
\end{definition}
Notice that for $S=[n]$, we recover the original definition of an inversion sequence of length $n$, i.e., $\I_{[n]}=\I_n$.
We continue to use the same notation for pattern avoidance on $S$-inversion sequences as on inversion sequences: $\I_S(\pi)$ is the set of $S$-inversion sequences that avoid $\pi$.
We note that we define $\I_{\emptyset}$ to contain the empty sequence, which avoids all patterns, consistent with the previous observation that $\I_{[n]}=\I_n$ when $n=0$.

Finally, define $e\cdot f$ as the concatenation of sequences $e$ and $f$. For example, $e\cdot f=142857$ for $e=14$ and $f=2857.$

\section{Wilf equivalences of length-4 patterns}\label{section: Wilf 4}

Before we prove the Wilf equivalences of length-4 patterns, we present the following useful result.

\begin{theorem}\label{theorem: 0 initial positive pattern S-inversion summation}
For any $n \geq 1$ and a pattern $\pi=\pi_1\cdots\pi_k$ where $\pi_1 = 0$ and $\pi_i > 0$ for all $i > 1$,
\[ |\I_n(\pi)| = \sum_{S \subseteq [n-1]} |\I_S(\pi_2\cdots\pi_k)|.\]
\end{theorem}
\begin{proof}
Define $\pi'=\pi_2\cdots\pi_k$.
Notice that all elements of $\pi'$ are positive.
Consider $e \in \I_n$, and let $S$ be the set of indices $i \in [2,n]$ such that $e_i > 0$.
We claim that $e_S$ avoids $\pi'$ if and only if $e$ avoids $\pi$.
If $e_S$ contains $\pi'$, then suppose $e_{i_1}\cdots e_{i_{k-1}} \sim \pi'$ for $i_1,\dots,i_{k-1}\in S$.
By definition of $S$, we have $e_{i_1},\dots,e_{i_{k-1}} > 0$, and thus adding in $e_1 = 0$ yields $e_1 e_{i_1} \cdots e_{i_{k-1}} \sim \pi$.
Conversely, if $e$ contains $\pi$, suppose $e_{i_0}\cdots e_{i_{k-1}}\sim \pi$ for $1 \leq i_0 < i_1 < \cdots < e_{i_{k-1}} \leq n$.
Then as all elements of $\pi$ are positive except for $\pi_1=0$, we have $e_{i_1},\dots,e_{i_{k-1}} > e_{i_0} \geq 0$, so $e_{i_1}\cdots e_{i_{k-1}}\sim \pi'$ is a subsequence of $e_S$, which thus contains $\pi'$.

Let $S^-=\{s-1\mid s \in S\}$.
Notice that $e_S^-$ is a $S^-$-inversion sequence, and there is a natural bijection between $\I_{S^-}$ and the elements of $\I_n$ with $e_i> 0$ if and only if $i \in S$.
Thus, for a fixed subset $S\subseteq[2,n]$, the number of elements of $\I_n(\pi)$ with $e_i>0$ if and only if $i \in S$ is equal to $|\I_{S^-}(\pi')|$.
Summing over all such subsets $S$ and re-indexing over $S^-$ instead yields the result.
\end{proof}

% The following two results were proven for inversion sequences, but we note that the same proofs work to obtain the following stronger results.
% \begin{lemma}[\protect{\cite[Observation 3]{corteel2016patterns}}]\label{lemma: 210 condition}
% Let $S$ be a finite set of positive integers.
% The $210$-avoiding $S$-inversion sequences are precisely those that can be partitioned into two weakly increasing subsequences.
% \end{lemma}
% \begin{lemma}[\protect{\cite[Observation 4]{corteel2016patterns}}]\label{lemma: 201 condition}
% Let $S$ be a finite set of positive integers, and let $n=|S|$.
% Let $e=e_1\cdots e_n \in \I_S$.
% For any $i \in [n]$, let $M_i = \max\{e_1,\dots,e_{i-1}\}$.
% Then $e \in \I_S(201)$ if and only if for every $i \in [n]$, either $i$ is a weak left-to-right maximum, or for all $j > i$, we have $e_j \leq e_i$ or $e_j > M_i$.
% \end{lemma}
% These two lemmas imply the following result, which was originally stated for inversion sequences rather than $S$-inversion sequences, but we note that the same proof works to obtain the following stronger result.
The following result was initially stated for inversion sequences rather than $S$-inversion sequences, but we note that the same proof works to obtain the following stronger result.
\begin{theorem}[\protect{\cite[Theorem 5]{corteel2016patterns}}]\label{theorem: 210 = 201 on S-inversion}
For any finite set $S$ of positive integers,
\[ |\I_S(210)|=|\I_S(201)|.\]
\end{theorem}

This allows us to prove that 0312 and 0321 are Wilf equivalent.
\begin{theorem}\label{theorem: 0312 = 0321}
For $n \geq 1$,
\[ |\I_n(0312)| = |\I_n(0321)| = \sum_{S \subseteq [n-1]} |\I_{S} (210)|. \]
\end{theorem}
\begin{proof}
The result follows from applying \cref{theorem: 210 = 201 on S-inversion} to \cref{theorem: 0 initial positive pattern S-inversion summation}.
\end{proof}

\subsection{Wilf equivalences by double induction}
% A similar argument allows us to prove that 0221 and 0212 are Wilf equivalent.
% To do so, we show that 110 and 101 are Wilf equivalent over $S$-inversion sequences.
The following lemma is useful for many of our later results.
A \emph{binary word} of length $n$ is an element of $\{0,1\}^n$, i.e., a string of $n$ zeros and ones.
Pattern avoidance on binary words is defined analogously.

\begin{lemma}\label{lemma: binary one zero on binary words}
Let $\pi=\pi_1\cdots\pi_\ell$ be a pattern of length $\ell \geq 2$ such that $\pi_i \in \{0,1\}$ for all $i$, and there exists exactly one $j$ such that $\pi_j=0$.
Then for any two integers $j,k \geq 0$, the number of binary words of length $j+k$ with $j$ zeros and $k$ ones that avoid $\pi$ is $\binom{j+\min\{k,\ell-2\}}{j}$.
\end{lemma}
\begin{proof}
There are $\binom{j+k}{k}$ binary words of length $j+k$ with $j$ zeros and $k$ ones, corresponding to choosing the positions of the ones.
If $k \leq \ell - 2$, as $\pi$ has $\ell-1$ ones, all of these binary words avoid $\pi$, so there are $\binom{j+k}{k}=\binom{j+\min\{k,\ell-2\}}{j}$ valid binary words.
If $k \geq \ell-1$, suppose $\pi_j$ for $j \in [\ell]$ is the unique zero in $\pi$.
If $j=1$ or $j=\ell$, then all but $\ell-2$ of the $k$ ones must be at the beginning or end of the binary word, respectively, which yields $\binom{j+\ell-2}{\ell-2}=\binom{j+\min\{k,\ell-2\}}{j}$ valid binary words.
Otherwise $2 \leq j \leq \ell-1$, and then the $i$-th and $(i+1)$-th ones of any valid binary word must be consecutive, for all $j-1 \leq i \leq k+j-\ell $.
If this were not the case, then suppose there exists such an $i$ where the $i$-th and $(i+1)$-th ones are not consecutive.
Then the $(i-j+2)$-th through $i$-th ones, a zero between the $i$-th and $(i+1)$-th ones, and the $(i+1)$-th through $(i+\ell-j)$-th ones, form a $\pi$ pattern.
Notice that $i-j+2\geq 1$ and $i+\ell-j \leq k$, so the indices are valid.
Note that this condition is a necessary and sufficient condition for the binary word to avoid $\pi$.
Thus, the $(j-1)$-th through $(k+j-\ell+1)$-th ones in the binary word are consecutive; viewing this block of $k-\ell+3$ ones as a single entity allows us to determine that there are $\binom{j+k-(k-\ell+2)}{j} = \binom{j+\min\{k,\ell-2\}}{j}$ such valid binary words.
\end{proof}

\cref{lemma: binary one zero on binary words} allows us to prove the following result.

\begin{theorem}\label{theorem: binary non-initial zero on S-inversion}
Let $\pi=\pi_1\cdots\pi_\ell$ and $\sigma=\sigma_1\cdots\sigma_\ell$ be two patterns of length $\ell \geq 3$ such that $\pi_i,\sigma_i\in\{0,1\}$ for all $2 \leq i \leq \ell$ and $\pi_1=\sigma_1=1$.
If there exists exactly one $j$ such that $\pi_j=0$ and exactly one $j'$ such that $\sigma_{j'}=0$, then for any finite set $S$ of positive integers, $|\I_S(\pi)|=|\I_S(\sigma)|$.
\end{theorem}
\begin{proof}
% The proof follows the same approach as that in \cref{theorem: 110 = 101 on S-inversion}.
Let $x_{S,j,k}$ denote the number of $\pi$-avoiding $S$-inversion sequences with $j$ zeros and $k$ ones, and similarly define $y_{S,j,k}$ for $\sigma$-avoidance.
We will prove the refinement that $x_{S,j,k}=y_{S,j,k}$ for all $S$, $j$, and $k$ by induction on $|S|$.

When $|S|< \ell$, the result trivially holds as all $S$-inversion sequences avoid all patterns of length $\ell$.
For the inductive step, assume the result holds for all $S$ with $|S|=n-1$; we will show the result holds for all $S$ with $|S|=n$ via a second induction on $\min S$.
For the base case $\min S = 1$, any $e\in\I_S$ has $e_1=0$, which cannot be part of a $\pi$ or $\sigma$ pattern, so $e=e_1\cdots e_n \in \I_S$ avoids $\pi$ if and only if $e_2\cdots e_n$ avoids $\pi$, and similarly for $\sigma$.
Hence, $x_{S,j,k}=x_{S\setminus\{1\},j-1,k}=y_{S\setminus\{1\},j-1,k}=y_{S,j,k}$.

Now assume the result holds for all $S$ of size $n$ with $\min S=m-1 \geq 1$; we will show the result holds for all $S$ with $S=m$.
Consider a given $S$ of size $n$ with $\min S = m \geq 2$, and let $S^- = \{s-1\mid s \in S\}$.
Define $\phi:\I_S\to\I_{S^-}$ by $\phi(e_1\cdots e_n)_i = \max\{e_i-1,0\}$.
Notice that $\pi$-avoidance and $\sigma$-avoidance are both preserved under $\phi$.

Consider a $\pi$-avoiding $S$-inversion sequence $e'$ with $j$ zeros and $k$ ones.
Then $\phi(e')$ has $k+j$ zeros.
We claim that for any $d \in \I_{S^-}(\pi)$ with $k+j$ zeros, there exist exactly $\binom{j+\min\{k,\ell-2\}}{j}$ sequences $e\in \I_S(\pi)$ with $j$ zeros and $k$ ones such that $\phi(e)=d$.
This would then show that
\[ x_{S,j,k}=\binom{j+\min\{k,\ell-2\}}{j}\sum_{i=0}^{n-k-j} x_{S^-,k+j,i}.\]
Consider some $d \in \I_{S^-}(\pi)$ with $k+j$ zeros.
As $\phi(e)=d$ and $e$ has $j$ zeros and $k$ ones, we find $e$ is completely determined apart from selecting which $k$ of the $k+j$ zeros in $d$ become ones in $e$.
As $d$ avoids $\pi$, we find that $e$ avoids $\pi$ if and only if the zeros and ones of $e$ avoid $\pi$.
By \cref{lemma: binary one zero on binary words}, we find there are $\binom{j+\min\{k,\ell-2\}}{j}$ such choices of zeros and ones of $e$ that avoid $\pi$, each yielding a distinct valid $e\in \I_S(\pi)$.

As $S^-$ satisfies the conditions of the inductive hypothesis, it now suffices to similarly show that for any $d \in \I_{S^-}(\sigma)$ with $k+j$ zeros, there exist exactly $\binom{j+\min\{k,\ell-2\}}{j}$ sequences $e \in \I_S(\sigma)$ with $j$ zeros and $k$ ones such that $\phi(e)=d$.
The argument is identical to that for $\pi$, as $\sigma$ also satisfies the conditions of \cref{lemma: binary one zero on binary words}.
This implies $x_{S,j,k}=y_{S,j,k}$ for all $j$ and $k$, and by induction, for all $S$. This completes the proof.
\end{proof}

\begin{corollary}\label{corollary: 1011 = 1101 = 1110 on S-inversion}
For any finite set $S$ of positive integers,
\[ |\I_S(1011)| = |\I_S(1101)| = |\I_S(1110)|.\]
\end{corollary}
This implies 1011, 1101, and 1110 are Wilf equivalent over inversion sequences.

\begin{corollary}\label{corollary: 0221 = 0212}
For $n \geq 1$,
\[ |\I_n(0221)| = |\I_n(0212)| = \sum_{S\subseteq [n-1]} |\I_S(110)|.\]
\end{corollary}
\begin{proof}
\cref{theorem: binary non-initial zero on S-inversion} implies $|\I_S(110)|=|\I_S(101)|$ for any finite set $S$ of positive integers, from which the result follows by applying \cref{theorem: 0 initial positive pattern S-inversion summation}.
\end{proof}

The following result augments the method of \cref{theorem: binary non-initial zero on S-inversion} to prove that $\pi\cdot\rho$ and $\sigma\cdot\rho$ are Wilf equivalent over $S$-inversion sequences, where $\pi$ and $\sigma$ satisfy the assumptions of \cref{theorem: binary non-initial zero on S-inversion} and $\rho$ consists only of twos.
\begin{theorem}\label{theorem: binary non-initial zero followed by twos on S-inversion}
Let $\rho=\rho_1\cdots\rho_h$ be a pattern of length $h \geq 0$ such that $\rho_i=2$ for all $i$.
Let $\pi=\pi_1\cdots\pi_\ell$ and $\sigma=\sigma_1\cdots\sigma_\ell$ be two patterns of length $\ell \geq 3$ such that $\pi_i,\sigma_i\in\{0,1\}$ for all $2 \leq i \leq \ell$ and $\pi_1=\sigma_1=1$.
If there exists exactly one $j$ such that $\pi_j=0$ and exactly one $j'$ such that $\sigma_{j'}=0$, then for any finite set $S$ of positive integers, $|\I_S(\pi\cdot \rho)|=|\I_S(\sigma\cdot \rho)|$.
\end{theorem}
\begin{proof}
When $h=0$, the result follows from \cref{theorem: binary non-initial zero on S-inversion}.

For convenience, define $\pi'=\pi\cdot\rho$ and $\sigma'=\sigma\cdot\rho$.
We use a similar double induction approach as in \cref{theorem: binary non-initial zero on S-inversion}.
% Let the number of \emph{terminal zeros} of an $S$-inversion sequence $e$ be the largest integer $t$ such that the last $t$ entries of $e$ are all zero.
Let the \emph{terminal $h$-repeat statistic} of an $S$-inversion sequence $e$ be the largest integer $r$ such that there are at least $r$ zeros in $e$, and letting $z$ denote the index of the $r$-th zero in $e$, then there exist positive integers $z<i_1<i_2<\cdots< i_h\leq |S|$ where $e_{i_1}=e_{i_2}=\cdots=e_{i_h}>0$; if no such $r$ exists, define the terminal $h$-repeat statistic to be 0.
For example, the terminal 1-repeat statistic is simply the number of non-terminal zeros, where a terminal zero only has zeros after it, if anything.

Let $x_{S,j,k,r}$ denote the number of $\pi'$-avoiding $S$-inversion sequences with $j$ zeros, $k$ ones, and terminal $h$-repeat statistic $r$, and similarly define $y_{S,j,k,r}$ for $\sigma'$-avoidance.
We will prove the refinement that $x_{S,j,k,r}=y_{S,j,k,r}$ for all $S,j,k,r$ by induction on $|S|$.

When $|S|<\ell+h$, the result trivially holds as all $S$-inversion sequences avoid all patterns of length $\ell+h \geq 3$.
For the inductive step, assume the result holds for all $S$ with $|S|=n-1$; we will show the result holds for all $S$ with $|S|=n$ via a second induction on $\min S$.
For the base case $\min S = 1$, any $e \in \I_S$ has $e_1=0$, which cannot be part of a $\pi'$ or $\sigma'$ pattern, so $e=e_1\cdots e_n \in \I_S$ avoids $\pi'$ if and only if $e_2\cdots e_n$ avoids $\pi'$, and similarly for $\sigma'$.
Hence, $x_{S,j,k,r}=y_{S,j,k,r}$ using the inductive hypothesis for $S\setminus\{1\}$.

Now assume the result holds for all $S$ of size $n$ with $\min S = m-1 \geq 1$; we will show the result holds for all $S$ with $S=m$.
Consider a given $S$ of size $n$ with $\min S = m \geq 2$, and let $S^-=\{s-1\mid s\in S\}$.
Define $\phi$ as in \cref{theorem: binary non-initial zero on S-inversion}, and notice that $\pi'$-avoidance and $\sigma'$-avoidance are both preserved under $\phi$.

Suppose $d \in \I_{S^-}(\pi')$, and consider the $S$-inversion sequences $e\in\I_S$ such that $\phi(e)=d$.
Similarly, suppose $d' \in \I_{S^-}(\sigma')$, and consider the $S$-inversion sequences $e'\in\I_S$ such that $\phi(e')=d'$.
Suppose $d$ and $d'$ both have $j+k$ zeros and terminal $h$-repeat statistic $r$.
By the inductive hypothesis, the number of such $d$ equals the number of such $d'$.

As $d$ and $d'$ both have $j+k$ zeros, $e$ and $e'$ must each have $j+k$ total zeros and ones.
Now restrict consideration to those $e$ and $e'$ that have $j$ zeros and $k$ ones.
As $d$ is $\pi'$-avoiding, $e$ avoids $\pi'$ if and only if no $\pi$ pattern occurs within its first $r$ zero and one entries.
Similarly, $e'$ avoids $\sigma'$ if and only if no $\sigma$ pattern occurs within its first $r$ zero and one entries.

Consider some choice of zeros and ones for the last $j+k-r$ zeros of $d$ and $d'$, i.e., some binary sequence in $\{0,1\}^{j+k-r}$.
We claim that the number of $e$ whose last $j+k-r$ zeros and ones follow this binary sequence equals the number of $e'$ whose last $j+k-r$ zeros and ones follow this binary sequence.
Notice that all such $e$ and $e'$ have the same terminal $h$-repeat statistic $r'$: if the binary sequence contains at least $h$ ones, then $r'$ is the number of zeros before the $h$-th-to-last one in $e$ or respectively $e'$; otherwise, $r'$ is the number of zeros within the first $r$ zeros and ones of $e$ or respectively $e'$.
As $e$ and $e'$ both have $k$ ones and $j$ zeros, and this binary sequence fixes the terminal $h$-repeat statistic, this would be a stronger refinement that implies $x_{S,j,k,r'}=y_{S,j,k,r'}$ for all $j,k,r'$.

Suppose this binary sequence has $j'$ zeros and $k'=j+k-r-j'$ ones, where we may assume $j' \leq j$ and $k' \leq k$, as otherwise no valid $e$ or $e'$, with $k$ ones and $j$ zeros, exist.
Hence both $e$ and $e'$ must have $j-j'$ zeros and $k-k'$ ones among the positions of the first $r$ zeros in $d$ and $d'$, respectively.
These zeros and ones in $e$ must avoid $\pi$, and these zeros and ones in $e'$ must avoid $\sigma$.
Then \cref{lemma: binary one zero on binary words} implies that the number of such $e$ equals the number of such $e'$, namely equaling $\binom{j-j'+\min\{k-k',\ell-2\}}{j-j'}$.

This implies $x_{S,j,k,r'}=y_{S,j,k,r'}$ for all $j,k,r'$, and by induction, for all $S$, which completes the proof.
\end{proof}
\begin{corollary}\label{corollary: 1012 = 1102 on S-inversion}
For any finite set $S$ of positive integers,
\[ |\I_S(1012)| = |\I_S(1102)|.\]
\end{corollary}
This implies 1012 and 1102 are Wilf equivalent over inversion sequences.

Similar to \cref{theorem: binary non-initial zero followed by twos on S-inversion}, which appends twos to $\pi$ and $\sigma$ that satisfy the assumptions of \cref{theorem: binary non-initial zero on S-inversion}, the following result augments the method of \cref{theorem: binary non-initial zero on S-inversion} to prove that $\rho\cdot\pi$ and $\rho\cdot\sigma$ are Wilf equivalent over $S$-inversion sequences, where $\pi$ and $\sigma$ satisfy the assumptions of \cref{lemma: binary one zero on binary words} and $\rho$ consists only of twos.

\begin{theorem}\label{theorem: twos followed by binary one zero on S-inversion}
Let $\rho=\rho_1\cdots\rho_h$ be a pattern of length $h \geq 1$ such that $\rho_i = 2$ for all $i$.
Let $\pi=\pi_1\cdots\pi_\ell$ and $\sigma=\sigma_1\cdots\sigma_\ell$ be two patterns of length $\ell \geq 2$ such that $\pi_i, \sigma_i \in \{0,1\}$ for all $i\in[\ell]$, and there exists exactly one $j$ such that $\pi_j=0$ and exactly one $j'$ such that $\sigma_{j'}=0$.
Then for any finite set $S$ of positive integers, $|\I_S(\rho\cdot\pi)|=|\I_S(\rho\cdot\sigma)|$.
\end{theorem}
\begin{proof}
For convenience, define $\pi'=\rho\cdot\pi$ and $\sigma'=\rho\cdot\sigma$.
We use an almost identical approach as in \cref{theorem: binary non-initial zero followed by twos on S-inversion}, reversing the definition of the terminal $h$-repeat statistic.
Let the \emph{initial $h$-repeat statistic} of an $S$-inversion sequence $e$ be the largest integer $r$ such that there are at least $r$ zeros in $e$, and letting $z$ denote the index of the $r$-th-to-last zero in $e$, then there exist positive integers $1 \leq i_1 < i_2 < \cdots < i_h < z$ where $e_{i_1}=e_{i_2}=\cdots=e_{i_h}>0$; if no such $r$ exists, define the initial $h$-repeat statistic to be 0.
For example, the initial 1-repeat statistic is simply the number of non-initial zeros, where an initial zero only has zeros before it, if anything.

Let $x_{S,j,k,r}$ denote the number of $\pi'$-avoiding $S$-inversion sequences with $j$ zeros, $k$ ones, and initial $h$-repeat statistic $r$, and similarly define $y_{S,j,k,r}$ for $\sigma'$-avoidance.
We will prove the refinement that $x_{S,j,k,r}=y_{S,j,k,r}$ for all $S,j,k,r$ by induction on $|S|$.

% Note: summarize the differences, to make the proof shorter
The proof then follows the same reasoning as that of \cref{theorem: binary non-initial zero followed by twos on S-inversion}.
For sake of brevity and clarity, we comment on some of the minor differences between the proofs.
We assume $h \geq 1$ so that neither $\pi'$ and $\sigma'$ start with a 0, allowing the base case $\min S = 1$ for the second induction to hold; this in turn allows us to lift the restriction that $\pi_1=\sigma_1=1$.
Using the same notation as in the proof of \cref{theorem: binary non-initial zero followed by twos on S-inversion}, the characterization of $e$ becomes as follows: $e$ avoids $\pi'$ if and only if no $\pi$ pattern occurs within its last $r$ zero and one entries, and similarly for $e'$ avoiding $\sigma'$.
We then consider some binary sequence for the first $j+k-r$ zeros of $d$ and $d'$, as opposed to the last, where fixing this binary sequence fixes the initial $h$-repeat statistic of $e$ and $e'$, so the proof proceeds identically.
\end{proof}

\begin{corollary}\label{corollary: 2011 = 2101 = 2110 on S-inversion}
For any finite set $S$ of positive integers,
\[ |\I_S(2011)| = |\I_S(2101)| = |\I_S(2110)|.\]
\end{corollary}

This implies 2011, 2101, and 2110 are Wilf equivalent over inversion sequences.

\begin{corollary}\label{corollary: 2201 = 2210 on S-inversion}
For any finite set $S$ of positive integers,
\[ |\I_S(2201)| = |\I_S(2210)|.\]
\end{corollary}

This implies 2201 and 2210 are Wilf equivalent over inversion sequences.

\begin{theorem}\label{theorem: 23 followed by binary one zero on S-inversion}
Let $\pi=\pi_1\cdots\pi_\ell$ and $\sigma=\sigma_1\cdots\sigma_\ell$ be two patterns of length $\ell \geq 2$ such that $\pi_i,\sigma_i\in\{0,1\}$ for all $i \in [\ell]$, and there exists exactly one $j$ such that $\pi_j=0$ and exactly one $j'$ such that $\sigma_{j'}=0$.
Then for any finite set $S$ of positive integers, $|\I_S(23\cdot\pi)| = |\I_S(23\cdot\sigma)|$.
\end{theorem}
\begin{proof}
For convenience, define $\pi'=23\cdot\pi$ and $\sigma'=23\cdot\sigma$.
We use a similar double induction approach as in \cref{theorem: twos followed by binary one zero on S-inversion}.

Let the \emph{initial non-inversion statistic} of an $S$-inversion sequence $e$ be the largest integer $z$ such that there are at least $z$ zeros in $e$, and there does not exist two elements $0 < e_{i_1} < e_{i_2}$ of $e$ where $i_1 < i_2$ and both come before the $z$-th zero in $e$; this statistic can equal zero.
Furthermore, let the \emph{initial positive set} of an $S$-inversion sequence $e$ be the set $P$ of integers $i$ for $0 \leq i < z$, where $z$ is the initial non-inversion statistic of $e$, such that there exists a positive element between the $i$-th and $(i+1)$-th zeros of $e$, where ``between the zeroth and first zeros of $e$" is interpreted to mean before the first zero of $e$.

Let $x_{S,j,k,z,P}$ denote the number of $\pi'$-avoiding $S$-inversion sequences with $j$ zeros, $k$ ones, initial non-inversion statistic $z$, and initial positive set $P$, and similarly define $y_{S,j,k,z,P}$ for $\sigma'$-avoidance.
We will prove the refinement that $x_{S,j,k,z,P}=y_{S,j,k,z,P}$ for all $S,j,k,z,P$ by induction on $|S|$.

The initial argument, from the base cases of $|S|<\ell+2$ up to the beginning of the second inductive step using $\phi$, follow the same reasoning as in \cref{theorem: twos followed by binary one zero on S-inversion}.
We use the same definitions for $S^-$ and $\phi$, where we note that $\pi'$-avoidance and $\sigma'$-avoidance are both preserved under $\phi$.

Suppose $d \in \I_{S^-}(\pi')$, and consider the $S$-inversion sequences $e \in \I_S$ such that $\phi(e)=d$.
Similarly, suppose $d' \in \I_{S^-}(\sigma')$, and consider the $S$-inversion sequences $e' \in \I_S$ such that $\phi(e')=d'$.
Suppose $d$ and $d'$ both have $j+k$ zeros, initial non-inversion statistic $z$, and initial positive set $P$.
By the inductive hypothesis, the number of such $d$ equals the number of such $d'$.

As $d$ and $d'$ both have $j+k$ zeros, $e$ and $e'$ must each have $j+k$ total zeros and ones.
Now restrict consideration to those $e$ and $e'$ that have $j$ zeros and $k$ ones.
As $d$ is $\pi'$-avoiding, $e$ avoids $\pi'$ if and only if no $\pi$ pattern occurs within its last $j+k-z$ zero and one entries.
Similarly, $e'$ avoids $\sigma'$ if and only if no $\sigma$ pattern occurs within its last $j+k-z$ zero and one entries.

Consider some choice of zeros and ones for the first $z$ zeros of $d$ and $d'$, i.e., some binary sequence in $\{0,1\}^z$.
We claim that the number of $e$ whose first $z$ zeros and ones follow this binary sequence equals the number of $e'$ whose first $z$ zeros and ones follow this binary sequence.
However, we first show that all such $e$ and $e'$ have the same initial non-inversion statistic $z'$ and initial positive set $P'$, as then the claim yields a stronger refinement that implies $x_{S,j,k,z',P'}=y_{S,j,k,z',P'}$ for all $j,k,z',P'$.

If the binary sequence contains no ones, then $z'=z$ and $P'=P$.
Otherwise, the ones in this binary sequence will cause $z' < z$ and may cause $P'$ to change.
If $P$ is empty, then notice $z'$ is simply the number of zeros in this binary sequence, and $P'$ is defined according to which zeros in the binary sequence have ones in between them.
If $P$ is non-empty, suppose $i_1$ is the index of the first one in the binary sequence; and let $i_2$ be the minimum element of $(P\cup\{z\})\cap[i_1,\infty)$.
Then $z'$ is the number of zeros within the first $i_2$ elements of the binary sequence, and $P'$ is uniquely determined from $P$ and the binary sequence.
Hence, given $z$, $P$, and the binary sequence, $z'$ and $P'$ are uniquely determined, as desired.

We conclude the proof by proving our claim that the number of such $e$ equals the number of such $e'$.
Suppose this binary sequence has $j'\leq j$ zeros and $k'=z-j' \leq k$ ones.
Then both $e$ and $e'$ have $j-j'$ zeros and $k-k'$ ones among the positions of the last $j+k-z$ zeros in $d$ and $d'$, respectively.
These zeros and ones in $e$ must avoid $\pi$, and these zeros and ones in $e'$ must avoid $\sigma$.
\cref{lemma: binary one zero on binary words} implies that the number of such $e$ equals the number of such $e'$, namely equaling $\binom{j-j'+\min\{k-k',\ell-2\}}{j-j'}$.

This implies $x_{S,j,k,z',P'}=y_{S,j,k,z',P'}$ for all $j,k,z',P'$, and by induction, for all $S$, which completes the proof.
\end{proof}

\begin{corollary}\label{corollary: 2301 = 2310 on S-inversion}
For any finite set $S$ of positive integers,
\[ |\I_S(2301)| = |\I_S(2310)|.\]
\end{corollary}

This implies 2301 and 2310 are Wilf equivalent over inversion sequences.

\subsection{Wilf equivalences by characterization}

For a sequence $e_1\cdots e_n$ of nonnegative integers, a position $j \in [n]$ is a \emph{weak left-to-right maximum} if $e_i \leq e_j$ for all $i < j$.
We use this definition to characterize 3210 and 3201-avoiding inversion sequences, allowing us to construct an explicit bijection between $\I_n(3210)$ and $\I_n(3201)$.
First, we characterize 3210-avoiding inversion sequences.
\begin{lemma}\label{3210 characterization}
The 3210-avoiding inversion sequences are precisely those that can be partitioned into three
weakly increasing subsequences.
\end{lemma}
\begin{proof}
Suppose $e\in\I_n$ has such a partition $e_{x_1} \le e_{x_2} \le \cdots \le e_{x_t}$, $e_{y_1} \le e_{y_2} \le\cdots \le e_{y_{r}}$, and $e_{z_1} \le e_{z_2} \le\cdots \le e_{z_{n-t-r}}$.
If there exist $i<j<k<\ell $ such that $e_i >e_j>e_k>e_{\ell}$, then no two of $i,j,k,\ell$ can both be in any of the three sets $\{x_1,\ldots, x_t\}$, $\{y_1,\ldots, y_r\}$, and $\{z_1,\ldots, z_{n-t-r}\}$, but this is impossible due to the Pigeonhole principle.
Therefore, $e$ avoids 3210.
Conversely, if $e$ is 3210-avoiding, let $x = (x_1,\ldots, x_t)$ be the sequence of weak left-to-right maxima of $e$.
Then $e_{x_1} \le e_{x_2} \le \cdots \le e_{x_t}$.
We then let $y = (y_1,\ldots, y_r)$ be the sequence of weak left-to-right maxima of the sequence obtained by deleting positions $\{x_1,\ldots, x_t\}$ from $e$, and in general we call these positions \emph{weak 2nd left-to-right maxima}.
Similarly $e_{y_1} \le e_{y_2} \le \cdots \le e_{y_r}$.
We then consider the remaining terms of the sequence, and take $i,j \not \in (\{x_1,\ldots, x_t\})\cup (\{y_1,\ldots, y_r\})$ where $i < j$.
The fact that $i$ is not included in $\{y_1,\ldots, y_r\}$ implies there exists some $v\in \{y_1,\ldots, y_r\}$ such that $v < i$ and $e_v >e_i$.
The fact that $v$ is not a weak left-to-right maxima implies there exists some $u$ such that $u < v$ and $e_u >e_v$.
Now we have $e_u>e_v>e_i$ with $u<v<i<j$.
Thus, to avoid 3210, we must have $e_i \le e_j$.
Both directions are thus concluded.
\end{proof}
Now, we characterize 3201-avoiding inversion sequences.
\begin{lemma}\label{3201 characterization}
Let $(e_1, e_2,\ldots , e_n) \in {\bf I}_n$. For any $i \in [n]$, let $M_i^1$ and $M_i^2$ be the largest and second largest value among $\{e_1, e_2,\ldots , e_{i-1}\}$, respectively. Then $e \in {\bf I}_n(3201)$ if and only if for every $i \in [n]$, the entry $e_i$ is either a weak left-to-right maximum, a weak 2nd left-to-right maximum, or for every $j>i$, we have $e_j \le e_i$  or $e_j \ge M_i^2$.
\end{lemma}
\begin{proof}
Let $e \in {\bf I}_n$ satisfy the conditions of \cref{3201 characterization} and, for the sake of contradiction, assume that there exists $i < j < k<\ell$ such that $e_k<e_{\ell} < e_j < e_i$. Notice that we have $M_k^1\ge e_i$ and thus $M_k^2\ge e_j$.
Then $e_k<e_{\ell}<e_j\leq M_k^2$, a contradiction to our assumption.

Conversely, suppose $e$ is 3201-avoiding. If $e_i$ is neither a weak left-to-right maximum nor a weak 2nd left-to-right maximum, then there exists some 2nd maximum value $M_i^2$ such that $M_i^2=e_v>e_i$ for some $v<i$. By definition of 2nd maximum, there is some maximum value $M_v^1=e_u>e_v$ for some $u<v$. To avoid 3201, we must have that for all $j>i$, $e_j\le e_i$ or $e_j\ge e_v=M_i^2$.
\end{proof}
Combining these two results allows us to prove $3210 \equiv 3201$.
\begin{theorem}\label{theorem: 3210 = 3201}
For $n \geq 1$, \[ |\I_n(3210)| = |\I_n(3201)|.\]
\end{theorem}

\begin{proof}
We exhibit a bijection based on the characterizations in \cref{3210 characterization} and \cref{3201 characterization}.

Given $e \in {\bf I}_n(3210)$, we define $f \in {\bf I}_n(3201)$ as follows. Let $e_{x_1} \le e_{x_2} \le \cdots \le e_{x_t}$ and $e_{y_1} \le e_{y_2} \le\cdots \le e_{y_{r}}$ be the subsequences of
weak left-to-right maxima and weak 2nd left-to-right maxima of $e$, respectively, and let $e_{z_1} \le e_{z_2} \le\cdots \le e_{z_{n-t-r}}$ be the remaining entries.

For $i \in [t]$, we set $f_{x_i}$ = $e_{x_i}$
and for $i \in [r]$, we set $f_{y_i}$ = $e_{y_i}$. For each $j \in [n-t-r]$, we extract an element of the multiset
$Z = \{e_{z_1}, e_{z_2}
,\ldots, e_{z_{n-t-r}}\}$ and assign it to $f_{z_1}, f_{z_2}
,\ldots, f_{z_{n-t-r}}$, one at a time in order, as follows: $$f_{z_j} := \max\{k \mid k \in Z - \{f_{z_1}, f_{z_2}
,\ldots, f_{z_{j-1}}\} \text{ and } k < M_{z_j}^2\}.$$
By definition, $f$ will satisfy the characterization property in \cref{3201 characterization} of ${\bf I}_n(3201)$.
One can see that this is invertible, hence a bijection.
\end{proof}

A computer search proves no other length-4 patterns are Wilf equivalent, except possibly 3012 and the aforementioned Wilf equivalence class $3210\equiv 3201$.
We leave this last case as an open question.

\begin{conjecture}\label{conjecture: 3210 = 3012}
For $n \geq 1$, \[ |\I_n(3201)| = |\I_n(3012)|.\]
\end{conjecture}

This has been verified for all $n \leq 12$.

\section{Enumeration of inversion sequences avoiding patterns of length 4}\label{section: enumeration 4}

Define a \emph{label-increasing tree} on $n$ vertices to be a rooted unordered tree in which each vertex is labeled with a distinct label from the set $\{0,\dots,n-1\}$ and labels increase along any path from the root to a leaf.
Then define a \emph{label-increasing tree with branching bounded by $k$} to be a label-increasing tree such that each vertex has at most $k$ children.
Let $L_{n,k}$ denote the set of $n$-vertex label-increasing trees with branching bounded by $k$.

Kuznetsov, Pak, and Postnikov \cite{kuznetsov1994increasing} showed that $L_{n,2}$ is in bijection with the \emph{up/down permutations}, that is, the permutations $\pi$ of $[n]$ such that $\pi_1 < \pi_2 > \pi_3 < \pi_4 > \cdots $; the number of up/down permutations is the Euler number $E_n$, whose exponential generating function is well-known, namely
\[ \sum_{n \geq 0} E_n \frac{x^n}{n!} = \tan(x)+\sec(x). \]
Corteel, Martinez, Savage, and Weselcouch \cite{corteel2016patterns} proved that $|\I_n(000)|=E_{n+1}$ via a bijection between $\I_n(000)$ and $L_{n+1,2}$.
We generalize their result to patterns $00\cdots 0$ of any length $k$.

\begin{theorem}\label{theorem: 00...0 label-increasing trees}
For $k\geq 1$, let $\pi=00\cdots0$ be the pattern consisting of $k$ zeros.
Then for all $n \geq 1$,
\[ |\I_n(\pi)| = |L_{n+1,k-1}|.\]
\end{theorem}
\begin{proof}
Notice that $\I_n(\pi)$ is the set of inversion sequences of length $n$ where each entry occurs at most $k-1$ times, and $L_{n+1,k-1}$ is the set of label-increasing trees of $n+1$ vertices labeled $0,\dots,n$, with branching bounded by $k$.
Then it is easy to see that the mapping sending $T\in L_{n+1,k-1}$ to $e \in \I_n(\pi)$, where $e_i$ is the parent of $i$ in $T$, is a bijection between $L_{n+1,k-1}$ and $\I_n(\pi)$.
\end{proof}

\cref{theorem: 00...0 label-increasing trees} implies $\I_n(0000)$ is in bijection with the label-increasing trees with branching bounded by 3, which is OEIS sequence A297196 \cite{sloane2018line}.
\cref{theorem: 00...0 label-increasing trees} also enables us to determine the exponential generating function for $|\I_n(00\cdots0)|$, as Riordan \cite{riordan1979forests} showed the exponential generating function
\begin{align}\label{eq: T_k(x) definition}
    T_k(x) = \sum_{n \geq 0} |L_{n,k}| \frac{x^n}{n!}
\end{align} 
satisfies the differential equation
\[ T'_k(x) = \sum_{i=0}^k \frac{(T_k(x)-1)^i}{i!}.\]
In other words, $T_k(x)$ satisfies $T_k(0)=1$ and 
\[ k!\,T_k'(x)=(T_k(x))^k+\sum_{m=0}^{k-2} c_{m,k}(T_k(x))^m,\]
where
\[ c_{m,k}=\frac{1}{m!}\left(\sum_{j=0}^{k-m}(-1)^{j}k(k-1)\cdots(j+1)\right).\]
These are the same coefficients satisfying 
\[ k! \sum_{j=0}^k\frac{x^j}{j!}=(x+1)^k+\sum_{m=0}^{k-2}c_{m,k}(x+1)^m \]
coming from the differential equation.

Let $L_{n,k}'$ denote the set of $n$-vertex label-increasing trees with unbounded root degree and branching bounded by $k$ at all other nodes. An alternative way to think about these combinatorial objects is to consider the possible ways how $n$ sufficiently large boxes can contain each other under the condition that each box may contain at most $k$ (themselves possibly nested) boxes.
Similar to \cref{theorem: 00...0 label-increasing trees}, we have the following result for patterns of the form $011\cdots1$.

\begin{theorem}\label{theorem: 01...1 label-increasing trees}
For $k\geq 1$, let $\pi=011\cdots1$ be the pattern consisting of a zero and $k$ ones.
Then for all $n \geq 1$,
\[ |\I_n(\pi)| = |L_{n+1,k-1}'|.\]
\end{theorem}
\begin{proof}
Notice that $\I_n(\pi)$ is the set of inversion sequences of length $n$ where each entry except $0$ occurs at most $k-1$ times, and $L_{n+1,k-1}'$ is the set of label-increasing trees of $n+1$ vertices labeled $0,1,\dots,n$, with branching bounded by $k$ except at the root.
Then it is easy to see that the mapping sending $T\in L_{n+1,k-1}'$ to $e \in \I_n(\pi)$, where $e_i$ is the parent of $i$ in $T$, is a bijection between $L_{n+1,k-1}'$ and $\I_n(\pi)$.
\end{proof}
\cref{theorem: 01...1 label-increasing trees} implies $\I_n(0111)$ is in bijection with the label-increasing trees (of unbounded root degree) with branching bounded by 2, which is OEIS sequence A000772 \cite{sloane2018line}.

More generally, it is well-established that $|L_{n,k}'|$ equals $D^n(\exp(x))$ evaluated at $x=0$, where the operator $D$ is defined by
\[ D = \left(\sum_{j=0}^k\frac{x^j}{j!}\right)\frac{d}{dx}. \]
Therefore, we have the general formulae of exponential generating function
\[ R_{k}(x):=\sum_{n\ge 0 }|L_{n,k}'|\frac{x^n}{n!}=\exp\big(T_k(x)-1\big), \]
where $T_k(x)$ is defined above in \cref{eq: T_k(x) definition}.
When $k=1$, the exponential generating function is $R_1(x)=\exp(\exp(x)-1)$, whose coefficients yield OEIS sequence A000110 \cite{sloane2018line}.
When $k=2$, the exponential generating function is
\[ R_2(x)=\exp\big(\tan(x) + \sec(x) - 1\big). \]
It is hard to explicitly write down $R_3(x)$, whose coefficients form OEIS sequence A094198.

Next, we present the following conjecture.

\begin{conjecture}
\label{conjecture: 0021 binomial transform}
Let $A_n=|\I_n(0021)|$. We have that
$A(x)=\sum_{n\ge 1}A_nx^n$ satisfies $$\frac{1} {\big(1 - A(x)\big)\big(1+A(x)\big)^2}  =  1 - x.$$
\end{conjecture}
In other words, $|\I_n(0021)|$ corresponds to the OEIS sequence A218225 \cite{sloane2018line}.
This has been verified for all $n \leq 11$.

\begin{remark}\label{remark:0021 solved}
Since the writing of the original version of this paper, \cref{conjecture: 0021 binomial transform} has been simultaneously proven by Chern, Fu, and Lin \cite{chern2022burstein} and Mansour \cite{mansour2022generating}.
\end{remark}

\section*{Acknowledgments}
The research was conducted at the 2021 University of Minnesota Duluth REU (NSF--DMS Grant 1949884 and NSA Grant H98230-20-1-0009) and fully supported by the generosity of the CYAN Mathematics Undergraduate Activities Fund.
The authors are deeply thankful to Professor Joseph Gallian for his long-lasting dedication in running the wonderful program.
The authors are also grateful to Amanda Burcroff for her editing feedback.
Lastly, we thank the anonymous referees who made many suggestions improving this article.

\bibliographystyle{plain}
\bibliography{ref}

\end{document}